\documentclass[a4paper, 10pt]{article}

\usepackage{amsthm}
\usepackage{amsmath}
\usepackage{amssymb}

\usepackage{todonotes}

\usepackage[shortlabels]{enumitem}

\usepackage{bussproofs}
\usepackage{multicol}

\usepackage{url}

\newtheorem{theorem}{Theorem}
\newtheorem{lemma}[theorem]{Lemma}
\newtheorem{corollary}[theorem]{Corollary}

\newtheorem*{consistency-lemma}{Lemma \ref{lem:d2-set-con-is1}}

\theoremstyle{definition}
\newtheorem{definition}[theorem]{Definition}
\newtheorem*{fact}{Fact}

\theoremstyle{remark}
\newtheorem*{remark}{Remark}

\newtheorem{question}{Question}

\newcommand{\N}{\mathbb{N}}

\newcommand{\X}{\mathcal{X}}

\newcommand{\eb}[2]{\exists #1 \! \le \! #2 \,}
\newcommand{\ab}[2]{\forall #1 \! \le \! #2 \,}
\newcommand{\es}[2]{\exists #1 \! < \! #2 \,}
\newcommand{\as}[2]{\forall #1 \! < \! #2 \,}

\newcommand{\RCA}{\mathrm{RCA}}
\newcommand{\ACA}{\mathrm{ACA}}
\newcommand{\WKL}{\mathrm{WKL}}
\newcommand{\RT}{\mathrm{RT}}

\newcommand{\ind}{\mathrm{I}}
\newcommand{\bd}{\mathrm{B}}
\newcommand{\PA}{\mathrm{PA}}

\newcommand{\IB}{\mathrm{IB}}

\newcommand{\RR}{\mathrm{R}}

\newcommand{\tuple}[1]{\langle#1\rangle}

\newcommand{\gn}[1]{\ulcorner#1\urcorner}

\newcommand{\Cod}{\mathrm{Cod}}
\newcommand{\Con}{\mathrm{Con}}
\newcommand{\Sat}{\mathrm{Sat}}
\newcommand{\Ack}{\mathrm{Ack}}

\title{How strong is Ramsey's theorem \\ if infinity can be weak?}
\author{Leszek Aleksander Ko{\l}odziejczyk\thanks{Institute of Mathematics, University of Warsaw, \texttt{\{lak, katarzyna.kowalik\}@mimuw.edu.pl}. Partially supported by grant no.~2017/27/B/ST1/01951 of the National Science Centre, Poland.}, \\ Katarzyna W.~Kowalik\footnotemark[1], \\ Keita Yokoyama\thanks{School of Information Science, Japan Advanced Institute of Science and Technology, \texttt{y-keita@jaist.ac.jp}. Partially supported by JSPS KAKENHI grant no.~19K03601.}}

\begin{document}
\maketitle

\begin{abstract}
We study the first-order consequences of Ramsey's Theorem for $k$-colourings of $n$-tuples, for fixed $n, k \ge 2$,
over the relatively weak second-order arithmetic theory $\RCA^*_0$. 
Using the Chong-Mourad coding lemma,
we show that in a model of $\RCA^*_0$ that does not satifsfy $\Sigma^0_1$ induction,
$\RT^n_k$ is equivalent to its relativization to any proper $\Sigma^0_1$-definable
cut, so its truth value remains unchanged in all extensions of the model with the same first-order universe.

We give a complete axiomatization of the first-order consequences of $\RCA^*_0 + \RT^n_k$ for $n \ge 3$.
We show that they form a non-finitely axiomatizable subtheory of $\PA$ whose
$\Pi_3$ fragment coincides with $\bd \Sigma_1 + \exp$ and whose
$\Pi_{\ell +3}$ fragment for $\ell \ge 1$ lies between 
$\ind \Sigma_\ell \Rightarrow \bd \Sigma_{\ell +1}$ and $\bd \Sigma_{\ell +1}$. 
We also give a complete axiomatization 
of the first-order consequences of $\RCA^*_0 + \RT^2_k + \neg \ind \Sigma_1$.
In general, we show that the first-order consequences of $\RCA^*_0 + \RT^2_k$
form a subtheory of $\ind \Sigma_2$ whose $\Pi_3$ fragment coincides
with $\bd \Sigma_1 + \exp$ and whose $\Pi_4$ fragment is strictly weaker
than $\bd \Sigma_2$ but not contained in $\ind \Sigma_1$.

Additionally, we consider a principle $\Delta^0_2$-$\RT^2_2$ which is defined
like $\RT^2_2$ but with both the $2$-colourings and the solutions allowed to be
$\Delta^0_2$-sets rather than just sets. We show that the behaviour of
$\Delta^0_2$-$\RT^2_2$ over $\RCA_0 + \bd\Sigma^0_2$ is in many ways 
analogous to that of $\RT^2_2$ over $\RCA^*_0$, and that 
$\RCA_0 + \bd \Sigma^0_2 + \Delta^0_2$-$\RT^2_2$
is $\Pi_4$- but not $\Pi_5$-conservative over $\bd \Sigma_2$.
However, the statement we use to witness failure of
$\Pi_5$-conservativity is not provable in $\RCA_0 +\RT^2_2$.
\end{abstract}

Over the last two decades, much of the research in reverse mathematics has concerned the logical strength of various
principles from Ramsey theory. One of the challenging problems in this area has been to 
characterize the first-order consequences of Ramsey's Theorem for pairs. 
Despite significant progress (e.g.~\cite{cholak-jockusch-slaman, csy:ind-strength-ramsey, py:rt22}), 
this remains open. In particular, it is not known 
whether Ramsey's Theorem for pairs and a fixed number of colours is $\Pi^1_1$ conservative
over the $\Sigma^0_2$ collection scheme.

In this paper, we study the first-order strength of Ramsey's Theorem -- both for pairs
and for longer tuples of fixed length -- over a weaker base theory than the one normally used
in reverse mathematics. Our base theory, $\RCA^*_0$,  differs from the usual system $\RCA_0$ 
in that the $\Sigma^0_1$ induction axiom of the latter is replaced by induction for bounded formulas only.

The study of $\RCA^*_0$ was initiated in \cite{simpson-smith} 
and continued in a number of later papers, 
e.g.~\cite{hatzikiriakou:algebraic-disguises, sy:peanocat, ky:categorical, enayat-wong}.
In the context of Ramsey theory, it is important that $\Sigma^0_1$ 
induction is needed to show that each infinite set has arbitrarily large finite subsets. Hence,
over $\RCA^*_0$ the infinite homogeneous sets witnessing various
principles might be so sparse that they have ``strictly smaller cardinality'' than $\N$,
so the principles can become weaker.
Indeed, Yokoyama \cite{yokoyama:rt-rca*0} showed that for each fixed $n,k$, $\RCA^*_0$ 
extended by Ramsey's Theorem for $n$-tuples and $k$ colours, $\RT^n_k$,
is $\Pi_2$-conservative over $\ind\Delta_0 + \exp$. 
We are able to go quite a bit beyond that result.

Recent work of Belanger \cite{belanger:coh} has demonstrated that the study of reverse mathematics over $\RCA^*_0$
is relevant to the traditional $\RCA_0$ framework as well.  
In fact, a large part of our original motivation for studying Ramsey's Theorem over $\RCA^*_0$
was the desire to understand whether it can help in understanding $\RT^2_2$ over $\RCA_0$.
The jury is still out on that. 
However, it has turned out that Ramsey theory in $\RCA^*_0$
is a highly interesting topic in its own right. 
It gives rise to new examples of principles that are
partially conservative but not $\Pi^1_1$-conservative over the base theory,
and it has intriguing connections to the model theory of first-order arithmetic.

After discussing the necessary background in a preliminary Section 1, we begin
the paper proper in Section 2 by proving that in models of $\RCA^*_0$ that
are \emph{not} models of $\RCA_0$, $\RT^n_k$ is equivalent
to its relativizations to $\Sigma^0_1$-definable cuts. 
One consequence of that result is that in some models of $\RCA^*_0$,
Ramsey's Theorem is computably true. This is not the case in the standard
model of arithmetic or in any other model of $\RCA_0$.

In Section 3, we use the equivalence from Section 2 to give an axiomatization
of the first-order consequences of $\RCA^*_0 +\RT^n_k$
where $n \ge 3$. In each case, this turns out
to be an unusual fragment of Peano Arithmetic that is $\Pi_3$- but
not $\Pi_4$-conservative over $\bd\Sigma_1+\exp$. 
Moreover, it is not contained in $\ind \Sigma_\ell$ for any $\ell$.

We then consider Ramsey's Theorem for pairs. We are not able to give a complete
axiomatization of its first-order consequences over $\RCA^*_0$, but 
in Section 4 we obtain some partial results. 
In particular, we do axiomatize these consequences over $\neg\ind\Sigma_1$.
We also show that $\RCA^*_0 + \RT^2_2$ is not conservative over (the lightface theory) $\ind\Sigma_1$. 

Then, in Section 5, we take a look at the question whether our results say
anything about Ramsey's Theorem for pairs over $\RCA_0$.
We consider a principle that can be viewed as a ``jumped version'' of $\RT^2_2$, 
and we show that it is not $\Pi_5$-conservative over $\RCA_0 + \bd\Sigma^0_2$. 
We also show that the most obvious sentence witnessing the lack of conservativity
is unprovable in $\RCA_0 +\RT^2_2$. 
However, the proof of unprovability, which is based on a possibly unexpected technique (proof speedup), 
no longer works for slightly weaker sentences.  

\section{Preliminaries}\label{sec:prelim}

We assume that the reader has some familiarity with fragments 
of second-order arithmetic, as described in \cite{simpson:sosoa} or \cite{hirschfeldt:slicing}. 
We also assume familiarity with some basic facts about first-order arithmetic
and its models -- most or all of the necessary information can be found 
in \cite{hirschfeldt:slicing}, and \cite{Kaye91} covers more than enough.

The symbol $\omega$ stands for the set of standard natural numbers. 
In contrast, $\N$ stands for the set of natural numbers as formalized in
the given theory we are studying -- in a nonstandard model, this
is the first-order universe of the model. 

Notation like $\Sigma^0_\ell$, $\Pi^0_\ell$ represents the usual formula classes defined
in terms of first-order quantifier alternations, but allowing second-order free variables.
On the other hand, notation without the superscript $0$, like $\Sigma_\ell$, $\Pi_\ell$,
represents analogously defined classes of first-order, or ``lightface'', formulas -- 
that is, without any second-order variables at all.
If we want to specify the second-order parameters appearing
in a $\Sigma^0_\ell$ formula, we use notation like $\Sigma_\ell(\bar X)$.
We extend these conventions to naming theories: thus, for example,
$\bd\Sigma^0_2$ is the fragment of second-order arithmetic
axiomatized by $\Delta^0_0$ induction and $\Sigma^0_2$ collection, 
whereas $\bd\Sigma_2$ is the fragment of first-order arithmetic
axiomatized by $\Delta_0$ induction and $\Sigma_2$ collection.

\begin{remark}
In formulating the results presented in the paper, we had to make the decision whether to state them in purely arithmetical, lightface, form, or in $\Pi^1_1$ form, allowing the appearance of (typically universally quantified) second-order parameters.
We opted to use the lightface version most of the time, with the tacit understanding that our results of the form ``first-order scheme $T$ implies first-order sentence $\psi$'' (as for instance Lemma \ref{lem:jockusch}) typically have a natural relativization of the form ``for all $X$, $T(X)$ implies $\psi(X)$'' that can be proved by essentially the same argument. On the other hand, we did allow second-order parameters whenever we found it advisable, for instance because it was necessary to state the result properly (as in Theorem \ref{thm:neg-is1-cs2}) or needed for later applications (as in the case of Theorem \ref{thm:rt-on-cut}). 
\end{remark}

Recall that for $\ell \ge 1$ the theory $\ind\Sigma_\ell$ proves (in fact, is equivalent to over $\ind \Delta_0$)
the scheme of \emph{strong $\Sigma_\ell$ collection}, that is,
\[\forall v \, \exists w\, \ab{x}{v} (\exists y\, \sigma(x,y) \Rightarrow \eb{y}{w} \sigma(x,y)),\]
where $\sigma(x,y)$ is a $\Sigma_\ell$ formula, possibly with parameters.

The theory $\RCA^*_0$ is obtained from $\RCA_0$ by weakening the $\ind\Sigma^0_1$ axiom
to $\bd\Sigma^0_1$ and adding the axiom $\exp$ that explicitly guarantees the totality of exponentiation.
The first-order consequences of $\RCA^*_0$ are axiomatized by $\bd\Sigma_1 + \exp$.

When we consider a model $(M,\X)$ of some fragment of second-order arithmetic 
(or simply work inside this fragment without reference to a specific model),
a \emph{set} is an element of the second-order universe, i.e.~an element of $\X$. In contrast,
a \emph{definable set} is any subset of $M$ that is definable in $(M,\X)$, but does not have to belong to $\X$.
A definable set is a \emph{$\Delta^0_\ell$-definable set}, or simply a \emph{$\Delta^0_\ell$-set} 
(resp., a \emph{$\Sigma^0_\ell$-definable set} or \emph{$\Sigma^0_\ell$-set}) 
if it happens to be definable by a $\Delta^0_\ell$ (resp.~$\Sigma^0_n$) formula. 
The notions of a \emph{$\Delta_\ell$-set} and \emph{$\Sigma_\ell$-set} are defined analogously.

Since most of the models we study only satisfy $\Delta^0_1$-comprehension, 
$\Delta^0_\ell$-sets for $\ell \ge 2$ and $\Sigma^0_\ell$-sets for $\ell \ge 1$ will not always be sets.
However, using appropriate universal formulas, 
we can quantify over $\Delta^0_\ell$- or over $\Sigma^0_\ell$-sets using second-order quantifiers 
(e.g.~``for every $X$, and every equivalent pair of a $\Sigma^0_\ell(X)$ and a $\Pi^0_\ell(X)$ formula, ...''). 
On the other hand, quantification over $\Delta_\ell$- or over $\Sigma_\ell$-sets is first-order.
We write $\Delta_{\ell}$-$\mathrm{Def}(M)$ (resp.~$\Delta^0_{\ell}$-$\mathrm{Def}(M,\X)$) for the
collection of $\Delta_\ell$-definable subsets of $M$ (resp.~the subsets of $M$ that are
$\Delta^0_\ell$-definable in $(M,\X)$).

For $\ell \ge 1$, let $\Sat_\ell(x,y)$ be the usual universal $\Sigma_\ell$ formula
and let $\Sat_\ell(x,y,X)$ be the usual universal $\Sigma^0_\ell$ formula with the unique second-order variable $X$.
Then $0^{(\ell)}$ is the $\Sigma_\ell$ definable set $\{e: \Sat_\ell(e,e)\}$; we write $0'$ for $0^{(1)}$. Similarly,
if $A$ is a set, then $A^{(\ell)}$ is $\{e: \Sat_\ell(e,e,A)\}$; this notion is generalized in a natural way to the case where
$A$ is merely a definable set. Note that $\bd \Sigma_\ell$ is enough to prove that 
$0^{(\ell+1)}$ and $(0^{(\ell)})'$ are mutually $\Delta_1$-definable.

For $n,k \in \omega$, $\RT^n_k$ stands for the usual formulation of Ramsey's Theorem for pairs
in second-order arithmetic: ``for every function $f \colon [\N]^n \to k$, there is an infinite
homogeneous set $H$ for $f$''. Importantly, ``$H$ is infinite'' is understood here as ``$H$ is unbounded'',
i.e.~for every $x \in \N$ there is $H \ni y \ge x$. If $\Sigma^0_1$ induction fails, this does not imply that
$H$ contains an $x$-element finite subset for every $x$. Ramsey's Theorem formulated in terms
of the latter notion is easily seen to imply $\ind \Sigma^0_1$ \cite{yokoyama:rt-rca*0}.

A \emph{cut} in a model of arithmetic $M$ is any subset $I \subseteq M$ 
which contains $0$ and is closed downwards and under successor; 
note that if $I \neq M$, it will never be a ``set'' in the sense of belonging 
to whatever second-order arithmetic structure there might be on $M$. 
A \emph{definable cut} is a cut that happens to be a definable set. 
If $(M,\X) \models \RCA^*_0$, and $I$ is a $\Sigma^0_1$-definable
cut in $M$, then there is an infinite set $A \in \X$ of cardinality $I$,
i.e.~$A = \{a_i: i \in I\}$ enumerated in increasing order.

For an element $s$ of a model $M$, $(s)_\Ack$ stands for $\{a \in M: M \models a \in_\Ack s\}$, where $\in_\Ack$
is the usual Ackermann interpretation of set theory in arithmetic 
(``the $a$-th bit in the binary notation for $s$ is 1''). 
Given a proper cut $I \subseteq M$, 
the collection $\Cod(M/I)$ of \emph{subsets of $I$ coded in $M$}
is $\{(s)_\Ack \cap I: s \in M\}$. 
If $M$ satisfies induction for any of the classes of formulas $\Gamma$ that we consider in this paper, 
this will coincide with $\{A \cap I: A \textrm{ a } \Gamma\textrm{-definable subset of } M\}$.

The collection $\Cod(M/\omega)$ is commonly referred 
to as the \emph{standard system} of $M$ and denoted by
$\mathrm{SSy}(M)$. The following combination 
of standard model-theoretic facts discussed in \cite{Kaye91}
and well-known results on $\RT^n_k$ presented e.g.~in \cite{hirschfeldt:slicing}
will often be used without notice.
\begin{fact}
Let $\mathcal{S} \subseteq \mathcal{P}(\omega)$ be such that $(\omega, \mathcal{S}) \models \WKL_0$
(such a family $\mathcal{S}$ is known as a \emph{Scott set}). 
If $\mathcal{S}$ is countable, then for every $\ell \ge 1$ there exists a model $M\models \bd\Sigma_{\ell}$ 
such that $\omega$ is $\Sigma_{\ell}$-definable in $M$ and $\mathrm{SSy}(M)=\mathcal{S}$.

For each fixed $n \ge 2$, there exist countable Scott sets $\mathcal{S}_1$ and $\mathcal{S}_2$ such that
$(\omega, \mathcal{S}_1)\models \RT^{n}_{2}$ and $(\omega, \mathcal{S}_1)\not\models \RT^{n}_{2}$.
\end{fact}

We will sometimes want to abuse notation and use $\Cod(M/I)$ for the collection of binary (as opposed to unary)
relations on $I$ coded in $M$, that is for $\{(s)_\Ack \cap \{\tuple{i,j}:i,j \in I\}: s \in M\}$ where $\tuple{\cdot,\cdot}$
is the usual Cantor pairing function. If $I$ is not closed under multiplication, then such binary relations 
might not be elements of $\Cod(M/I)$ in the strict sense, but that should not lead to any confusion.

We define the iterated exponential function $\exp_n(x)$ by: $\exp_0(x) = x$, and $\exp_{n+1}(x) = 2^{\exp_n(x)}$.

\section{Characterization in terms of cuts}\label{sec:cuts}

In this section, we prove a basic result which underlies 
our subsequent analysis of Ramsey's Theorem over $\RCA^*_0$: 
if $\Sigma^0_1$ induction fails but $\Sigma^0_1$ collection holds, 
then Ramsey's Theorem is equivalent to its own
relativization to a proper $\Sigma^0_1$-definable cut.
To prove this, we make use of an important fact 
about coding sets in models of collection.
 
\begin{lemma}[\cite{chong-mourad:degree-cut}] 
Let $(M,\X) \models \RCA^*_0 + \bd\Sigma^0_n$.
Then for every pair of bounded disjoint $\Sigma^0_n$-definable sets $X,Y \subseteq M$
there exists $A \in \X$ such that $A \cap (X \cup Y) = X$. 
\end{lemma}

\begin{corollary}\label{cor:chong-mourad-cut}
Let $(M,\X) \models \RCA^*_0 + \bd\Sigma^0_n$ and let 
$I \subseteq M$ be a proper cut in $M$.
If $X \subseteq I$ is such that both $X$ and $I \setminus X$ are $\Sigma^0_n$-definable sets,
then $X \in \Cod(M/I)$.
\end{corollary}

\begin{theorem}\label{thm:rt-on-cut}
Let $(M, \X) \models \RCA^*_0$ and let $I \subseteq M$ be 
a $\Sigma^0_1$-definable proper cut in $M$. Then for every $n,k \in \omega$:
\begin{equation}\label{eqn:ramsey-characterization}
(M,\mathcal{X}) \models \mathrm{RT}^n_k \textrm{ iff } (I,\mathrm{Cod}(M/I)) \models \mathrm{RT}^n_k.
\end{equation}
\end{theorem}

\begin{proof}
Let $(M, \mathcal{X})$ be a model of $\RCA^*_0$ and let $I \subseteq M$ be 
a $\Sigma^0_1$-definable proper cut. Let $A \in \X$ be an infinite subset of $M$
which can be enumerated in increasing order as~$\{a_i:i \in I\}$. 
We may assume that $0 \in A$. Fix standard $n, k$.

Suppose $(M, \mathcal{X}) \models \RT^n_k$. Let $f \colon [I]^n \to k$
be coded by $c \in M$. We can use $f$ to 
define a colouring $\check{f} \colon [A]^n \to k$ in the following way:
\[\check{f}(a_{i_1},\ldots,a_{i_n}) = f(i_1,\ldots,i_n).\]
In fact, it is easy to generalize the definition of $\check{f}$ to obtain a colouring of $[M]^n$,
which we will continue to call $\check{f}$:
\begin{equation*}
\check{f}(x_1,\ldots,x_n) = 
\begin{cases}
f(i_1,\ldots,i_n) & \textrm{if } i_1 < \ldots < i_n \in I \textrm{ are such that } \\
                       &    x_1 \in [a_{i_1},a_{i_1+1}), \ldots, 
                           x_n \in [a_{i_n},a_{i_n+1}), \medskip \\ 
0 & \textrm{if there are no such } i_1,\ldots,i_n.
\end{cases}
\end{equation*}

Note that $\check{f}$ is $\Delta_1(A,c)$-definable, so $\check{f} \in \X$. By $\RT^n_k$, there exists
an infinite $H \in \X$ homogeneous for $\check{f}$. By Corollary \ref{cor:chong-mourad-cut},
the $\Sigma_1(H)$-definable set 
\[\hat H = \{i \in I: H \cap [a_i,a_{i+1}) \neq \emptyset\}\] 
is in $\Cod(M/I)$. Clearly, $\hat H$ is cofinal in $I$ and homogeneous for $f$.

In the other direction, suppose $(I, \Cod(M/I)) \models \RT^n_k$. 
Consider a colouring $f \colon [M]^n \to k$. By Corollary \ref{cor:chong-mourad-cut},
the colouring $\hat f \colon [I]^n \to k$ given by
\[\hat{f}(i_1,\ldots,i_n) = f(a_{i_1},\ldots,a_{i_n}) \]
is in $\Cod(M/I)$. Since $(I, \Cod(M/I)) \models \RT^n_k$,
there is $\Cod(M/I) \ni H \subseteq I$ cofinal in $I$ and homogeneous for $\hat f$. 
Then the set $\check H = \{a_i: i \in H\}$ is in $\X$ and it is an infinite
subset of $M$ homogeneous for $f$. 
\end{proof}

\begin{remark}
Note that the left-hand side of the equivalence (\ref{eqn:ramsey-characterization}) 
in Theorem \ref{thm:rt-on-cut} does not depend
on the choice of the cut $I$, while the right-hand side does not depend on $\X$, 
as long as $I$ is $\Sigma^0_1$-definable in $(M,\X)$.
Thus, Theorem \ref{thm:rt-on-cut} means that over $\RCA^*_0$, once $\ind\Sigma^0_1$ fails,
Ramsey's Theorem becomes in some sense a first-order property.
In particular, it can be satisfied in some structures of the form 
$(M, \Delta_1$-$\mathrm{Def}(M))$
(``computably true in $M$''). We investigate this phenomenon
further in the next two sections of the paper.
\end{remark}

\section{Ramsey for triples and beyond}\label{sec:rt32}

We now use the characterization provided by Theorem \ref{thm:rt-on-cut}
to study the first-order consequences of $\RCA^*_0 + \RT^n_k$ for $n \ge 3$. 
We begin with the easy but useful observation that, just like over $\RCA_0$, the strength 
of Ramsey's Theorem for $n$-tuples does not increase if we consider a larger but fixed 
number of colours. 

\begin{lemma}
For each $n, k \ge 2$, $\RCA^*_0 \vdash (\RT^n_k \Leftrightarrow \RT^n_{k+1})$.
\end{lemma}
\begin{proof}
Assume $\RCA^*_0 + \RT^n_k$ and let $f \colon [\N]^n \to k+1$.
Consider the colouring $g \colon [\N]^n \to k$ given by $g(\bar x) = \min(f(\bar x),k-1)$.
Let $A$ be an infinite homogeneous set for $g$ and let $\{a_i: i \in I\}$ 
be an increasing enumeration of $A$. (Here $I$ may be either a proper $\Sigma^0_1$-definable cut or $\N$,
depending on $A$.)

If $A$ is $j$-homogeneous for $g$ with $j < k-1$, then $A$ is also $j$-homogeneous for $f$,
so we are done. Otherwise, $A$ is $(k-1)$-homogeneous for $g$,
which means that $f{\upharpoonright}_{[A]^n}$ takes at most the two values
$k-1$ and $k$. Define a $2$-colouring of $[\N]^n$ by:
\begin{equation*}
\check{f}(x_1,\ldots,x_n) = 
\begin{cases}
f(a_{i_1},\ldots,a_{i_n}) - k + 1 & \textrm{if } i_1 < \ldots < i_n \in I \textrm{ are such that } \\
                       &    x_1 \in [a_{i_1},a_{i_1+1}), \ldots, 
                           x_n \in [a_{i_n},a_{i_n+1}), \medskip \\ 
0 & \textrm{if there are no such } i_1,\ldots,i_n.
\end{cases}
\end{equation*}
Let $H$ be an infinite homogeneous set for $\check{f}$. Then the set \[H':=\{a_i: i \in I \textrm{ and } H \cap [a_i,a_{i+1}) \neq \emptyset\}\] exists 
by $\Delta^0_1$-comprehension: it is clearly $\Sigma^0_1$-definable, and its
complement is the union of $\N \setminus A$ and the $\Sigma^0_1$-definable set 
$\{a_i:  \exists a \! \in \! A\,( a > a_i \textrm{ and } H \cap [a_i,a) = \emptyset)\}$. Moreover,
$H'$ is infinite and homogeneous for $f$.
\end{proof}

\begin{definition}
For $\ell \ge 1, n,k \ge 2$, let $\Delta_{\ell}$-$\RT^n_k$ be the first-order statement: 
``for every $\Delta_\ell$-definable $k$-colouring of $[\N]^n$,
there is a $\Delta_\ell$-definable infinite homogeneous set''.

Thus, a model $M$ satisfies $\Delta_{\ell}$-$\RT^n_k$ exactly if
$(M,\Delta_{\ell}$-$\mathrm{Def}(M)) \models \RT^n_k$. 
\end{definition}

It is well known that each $\Delta_{\ell}$-$\RT^n_k$ is false in the standard model. However, the usual
argument makes use of a nontrivial amount of induction.

\begin{lemma}\label{lem:jockusch}
For each $n \ge 2$:
\begin{enumerate}[(a)]
\item\label{it:jockusch-1} $\ind \Sigma_1$ proves that there is a $\Delta_1$-definable $2$-colouring of $[\N]^n$ with no $\Sigma_1$-definable infinite homogeneous set, 
\item\label{it:jockusch-general} for each $l \ge 1$, $\ind \Sigma_{\ell +1}$ proves that there is a $\Delta_{\ell}$-definable $2$-colouring of $[\N]^n$ with no $\Sigma_{\ell+1}$-definable infinite homogeneous set. 
\end{enumerate}
\end{lemma}
\begin{proof}Clearly, it is enough to prove the statement for $n = 2$.

The proof of \ref{it:jockusch-general} is just a formalization
of the usual proof due to \cite{jockusch:ramsey} in $\ind\Sigma_{\ell+1}$.
The place where $\Sigma_{\ell +1}$-induction is used is when
we are given a hypothetical $\Delta_{\ell+1}$-definable infinite homogeneous 
set with code $e$, and we want to reach a contradiction by looking at the
first $2e+2$ elements of this set. To do this, we need to know that
the set actually has at least $2e+2$ elements, and this is justified by proving
``for every $x$, the $\Delta_{\ell+1}$-set with code $e$ has a finite subset with at least $x$ elements''
by induction on $x$.

To prove \ref{it:jockusch-1}, one could formalize Specker's construction \cite{specker:ramsey} 
of a computable $2$-colouring of pairs with no r.e.~homogeneous set within $\ind \Sigma_1$.
Instead of that, we choose to formalize a weaker variant of the argument of \cite{jockusch:ramsey} 
proving \ref{it:jockusch-general} for $\ell = 1$.
We define a computable function $f \colon [\N]^2 \to 2$ in the following way. 
At stage $s$, we determine the values $f(n,s)$ for $n < s$.
To do this, we consider all $\Sigma_{1}$ formulas with codes $0,\ldots,\lfloor (s-1)/2\rfloor$. 
Given $e \le \lfloor (s-1)/2\rfloor$, 
if $e$ is the code of a $\Sigma_1$ formula $\exists v\, \delta(x,v)$ 
and there are at least $2e+2$ elements $x < s$ such that
$\eb{v}{s}\mathrm{Sat}_0(\gn{\delta},(x,v))$ holds, 
then choose the smallest two such elements $x_0,x_1$ for which $f(x_0,s), f(x_1,s)$ 
have not yet been defined, and let $f(x_i,s)=i$. Otherwise, do nothing.  
Once all the formulas with codes $0,\ldots, \lfloor (s-1)/2\rfloor$ have been dealt with,
complete stage $s$ by letting $f(x,s) = 0$ for all those $x<s$ for which
$f(x,s)$ was not defined earlier.

Now if the formula $\exists v\, \delta(x,v)$ with code $e$ defines 
an infinite homogeneous set for $f$, we can use $\Sigma_1$ induction
to conclude that there are at least $2e + 2$ elements $x$ 
such that $\exists v\, \delta(x,v)$ holds.
Consider the $2e + 2$ smallest such elements, say $x_0 <\ldots < x_{2e+1}$.
By another application of $\Sigma_1$ induction,
there is some $s > \max(2e, x_{2e+1})$
such that for $x \le x_{2e+1}$, 
if $\exists v\, \delta(x,v)$, then $\eb{v}{s} \delta(x,v)$.
Since there are infinitely many elements $x$ such that $\exists v\, \delta(x,v)$,
we can also assume that $\exists v\, \delta(s,v)$.
But the lower bounds on $s$ imply that at stage $s$ there will be some $i < j \le 2e+1$
such that $\exists v\, \delta(x_i,v), \exists v\, \delta(x_j,v)$,
and $f(x_i,s) \neq f(x_j,s)$. This is a contradiction, because all three elements $x,x',s$
satisfy a formula that defines a homogeneous set for $f$.
\end{proof}

\begin{lemma}\label{lem:rt-gives-jump}
Let $(M,\X) \models \RCA^*_0 + \RT^n_2$ where $n \ge 3$ and assume that $M \models \ind\Sigma_\ell$.
Then $0^{(\ell)} \in \X$. As a consequence, $\Delta_{\ell+1}$-$\mathrm{Def}(M) \subseteq \X$ and $M \models \bd \Sigma_{\ell +1}$.
\end{lemma}
\begin{proof}
Let $M \models \RCA^*_0 + \RT^n_2 + \ind \Sigma_\ell$.
We will prove by induction on $j \le \ell$ that $0^{(j)} \in \X$.
For $j = \ell$, this will immediately imply $\Delta_{\ell+1}$-$\mathrm{Def}(M) \subseteq \X$ and $M \models \bd \Sigma_{\ell+1}$ because $(M,\X)$ satisfies $\Delta^0_1$ comprehension and $\bd\Sigma^0_1$.

The base step of the induction holds by $\Delta^0_1$-comprehension in $(M,\X)$.
So, let $j < \ell$ and assume that $0^{(j)} \in \X$. 
We have to prove that $0^{(j+1)} \in \X$.

Consider the usual computable instance of $\RT^3_2$ whose solutions compute $0'$ and relativize
it to $0^{(j)}$:
\begin{equation*}
f(x,y,z) =
\begin{cases}
0 & \textrm{if there is a } \Sigma_{j+1} \textrm{ sentence } \exists v\, \pi(v) \textrm{ with code at most } x \\
  & \textrm{such that } \ab{v}{y}\mathrm{Sat}_j(\gn{\neg \pi},v) \land \eb{v}{z}\neg\mathrm{Sat}_j(\gn{\neg \pi},v), \\
1 & \textrm{otherwise.}
\end{cases}
\end{equation*}
The colouring $f$ is $\Delta_1(0^{(j)})$-definable, so $f \in \X$. By $\RT^n_2$,
there exists an infinite $H \in \X$ homogeneous for $f$. 
We claim that $H$ cannot be $0$-homogeneous for $f$.
To see this, note that by $\ind \Sigma_\ell$ 
we have strong $\Sigma_{j+1}$ collection, 
so for any given $x$ there is a bound $w$ 
such that for any $\Sigma_{i+1}$ sentence with code below $x$, 
if the sentence is true, then there is a witness for it below $w$. 
Thus, for any $z > y \ge w$, we must have
$f(x,y,z) = 1$, which implies that no infinite set can be $0$-homogeneous for $f$.

So, $H$ is $1$-homogeneous for $f$.
We can now compute $0^{(j+1)}$ with oracle access to $0^{(i)} \oplus H$ as follows:
given a $\Sigma_{j+1}$ sentence $\exists v\, \pi(v)$, 
find some $x \in H$ above the code for the sentence,
find $y \in H$ above $x$, and use $0^{(j)}$ to determine whether $\eb{v}{y}\pi(v)$ holds;
if is does not, then neither does $\exists v\, \pi(v)$. 
Both $0^{(j)}$ and $H$ are in $\X$, so $0^{(j+1)} \in \X$ as well.
\end{proof}

We are now ready to give an axiomatization of the first-order part of $\RCA^*_0 +\RT^n_2$ for $n\ge 3$.
Afterwards, we will study the relationship of this theory to the usual fragments of first-order arithmetic.

\begin{theorem}\label{thm:rt32-axioms}
Let $n \ge 3$ and let $\RR^n$ be the theory:
\begin{equation}\label{eqn:theory-rt32} \left\{  (\bd \Sigma_{\ell+1} \land \exp) \lor 
\bigvee_{j = 1}^{\ell} \Delta_j\textrm{-}\RT^n_2: 
\ell \in \omega \right \} .
\end{equation}
Then $\RR^n$ axiomatizes the first-order consequences of $\RCA^*_0 + \RT^n_2$.
\end{theorem}

\begin{proof}
Fix $n \ge 3$ and let $\RR^n$ be as in (\ref{eqn:theory-rt32}).

We first argue that for every $M \models \RR^n$ there is a family of sets $\X \subseteq \mathcal{P}(M)$ such that
$(M, \X) \models \RCA^*_0 + \RT^n_2$,
which will mean that $\RR^{n}\not\models\psi$ implies $\RCA^{*}_{0}+\RT^{n}_{2}\not\models \psi$
for each arithmetical sentence $\psi$.
So, let $M \models \RR^n$. If $M \models \mathrm{PA}$, then
$(M, \mathrm{Def}(M))$ is a model of $\ACA_0$ and, \emph{a fortiori}, of $\RCA^*_0 + \RT^n_2$.

Otherwise, let $\ell \in \omega$ be the smallest such that $M \models \neg \ind \Sigma_{\ell+1}$.
For each $j = 1,\ldots,\ell$, it follows from Lemma \ref{lem:jockusch} that there is a $\Delta_j$-definable 
$2$-colouring of $[M]^n$ with no $\Delta_j$-definable homogeneous set, so $\RR^n$ implies that
$\bd \Sigma_{\ell+1} + \exp$ must hold in $M$. 
Moreover, since $\bd \Sigma_{\ell+2}$ fails, it must be the case that
$M \models \Delta_{\ell+1}$-$\RT^n_2$.
Thus $(M,\Delta_{\ell+1}$-$\mathrm{Def}(M)) \models \RCA^*_0 + \RT^n_2$.

In the other direction, we assume that $(M,\X) \models \RCA^*_0 + \RT^n_2$
and prove that $M \models \RR^n$.
This is clear if $M\models\PA$.
Otherwise, let $\ell$ be such that $M \models \neg \bd \Sigma_{\ell+1}$. 
Let $j \le \ell$ be the largest such that $M \models \ind \Sigma_j$.
By Lemma \ref{lem:rt-gives-jump}, $M \models \bd \Sigma_{j+1}$,
so in particular $j < \ell$. Moreover, $\Delta_{j+1}$-$\mathrm{Def}(M) \subseteq \X$. 
We now argue that $(M, \Delta_{j+1}$-$\mathrm{Def}(M)) \models \RT^n_2$, 
which will complete the argument.

Let $I$ be a $\Sigma_{j+1}$-definable proper cut in $M$. 
The cut $I$ is $\Sigma^0_1$-definable in 
$(M, \Delta_{j+1}$-$\mathrm{Def}(M))$ and thus also in $(M,\X)$.
Moreover, both of these structures satisfy $\RCA^*_0$. 
Therefore, Theorem \ref{thm:rt-on-cut} and the fact
that $(M,\X) \models \RCA^*_0 + \RT^n_2$
let us conclude that $(M, \Delta_{j+1}$-$\mathrm{Def}(M)) \models \RT^n_2$
as well.
\end{proof}

\begin{definition}
The theory $\IB$ is axiomatized by $\bd\Sigma_1$ and the set of sentences
\[\{\ind\Sigma_\ell \Rightarrow \bd\Sigma_{\ell+1}: \ell \ge 1\}.\]
\end{definition}

Kaye \cite{kaye:constructing-kappa-like} showed that $\IB + \exp$ implies the theory of all $\kappa$-like models of arithmetic
(for $\kappa$ possibly singular). 
It is now known (see \cite[Section 3.3]{haken:thesis}, \cite[Section 6]{bcwwy:php-konig})
that  $\IB + \exp$ is actually strictly stronger than the theory of all $\kappa$-like models.

\begin{theorem}\label{thm:rt32-fragments-pa} Let $n \ge 3$. Then:
\begin{enumerate}[(a)]
\item\label{it:rt32-fo} the first-order consequences of $\RCA^*_0 + \RT^n_2$ are strictly in between $\IB + \exp$
and $\PA$; as a result, they are not finitely axiomatizable.
\item\label{it:rt32-fo-levels} the $\Pi_3$ consequences of $\RCA^*_0 + \RT^n_2$ coincide with $\bd\Sigma_1 + \exp$;
for $\ell \ge 1$, the $\Pi_{\ell +3}$ consequences are strictly in between \[\bd\Sigma_1 +{\exp} + \bigwedge_{1 \le j \le \ell}(\ind\Sigma_j \Rightarrow \bd\Sigma_{j+1})\] and $\bd\Sigma_{\ell+1}$.
\end{enumerate}
\end{theorem}
\begin{proof}
We first prove \ref{it:rt32-fo-levels}. As in Theorem \ref{thm:rt32-axioms},
we let $\RR^n$ stand for the first-order consequences of $\RCA^*_0 + \RT^n_2$.

It follows immediately from the definition of $\RCA^*_0$ and Lemma \ref{lem:rt-gives-jump}
that the $\Pi_{\ell +3}$ consequences of $\RR^n$ include
$\bd\Sigma_1 + {\exp}$ and $\ind\Sigma_\ell \Rightarrow \bd\Sigma_{\ell+1}$ for each $j \le \ell$.
For $\ell \ge 1$, the inclusion is strict, because the statement
\[(\bd \Sigma_{\ell+1} \land \exp) \lor \bigvee_{j = 1}^{\ell} \Delta_j\textrm{-}\RT^n_2\]
is $\Pi_{\ell+3}$ but not provable 
in $\bd\Sigma_1 +{\exp} + \bigwedge_{1 \le j \le \ell}(\ind\Sigma_j \Rightarrow \bd\Sigma_{j+1})$.
To see the unprovability, consider a model $M \models \bd\Sigma_\ell + \exp$ such that
$\omega$ is $\Sigma_{\ell}$-definable in $M$ and $(\omega,\mathrm{SSy}(M)) \not \models \RT^n_2$.
Then, clearly, $M \models \ind\Sigma_j \Rightarrow \bd\Sigma_{j+1}$ for each $j \le \ell$; 
in fact, $M$ is a model of $\IB$.
However, Lemma \ref{lem:jockusch} 
implies that $(M,\Delta_{j}$-$\mathrm{Def}(M)) \not \models \RT^n_2$
for each $1 \le j \le \ell-1$.
On the other hand, $(M,\Delta_{\ell}$-$\mathrm{Def}(M))$ is a model of
$\RCA^*_0$ in which $\omega$ is $\Sigma^0_1$-definable,
so by Theorem \ref{thm:rt-on-cut} and the choice of $\mathrm{SSy}(M)$
it does not satisfy $\RT^n_2$ either.

Using a model $M$ chosen similarly but with $(\omega,\mathrm{SSy}(M)) \models \RT^n_2$,
we get $(M,\Delta_{\ell}$-$\mathrm{Def}(M)) \models \RT^n_2 + \neg \bd \Sigma_{\ell +1}$.
Thus, $\RR^n$ does not prove $\bd \Sigma_{\ell +1}$ for $\ell \ge 1$.

To see that all $\Pi_{\ell + 3}$ consequences of $\RR^n$ follow
from $\bd\Sigma_{\ell+1}$ for $\ell \ge 1$
let the $\Sigma_{\ell + 3}$ formula
$\psi := \exists x\, \forall y\, \exists z\, \pi(x,y,z)$ be consistent with $\bd\Sigma_{\ell+1}$,
let $K \models \bd\Sigma_{\ell+1} \land \psi$
be such that $(\omega,\mathrm{SSy}(K)) \models \RT^n_2$, 
and let $a \in K$ be a witness for the initial existential quantifier in $\psi$. 
By $\bd\Sigma_{\ell+1}$, the function
\begin{align*}
f(y) & =\textrm{least } w > y \textrm{ such that } \ab{y'}{y}\eb{z}{w}\pi(a,y',z) \\
     & ~~~~\textrm{and ``true }\Sigma_\ell \textrm{ sentences with codes } {\le y} \textrm{ are witnessed } {\le w}\textrm{''}
\end{align*}
is total and $\Delta_{\ell+1}$-definable in $K$.
Let $M$ be the cut $\sup_K(\{f^{m}(a): m \in \omega\})$.
Then $M \models \bd\Sigma_{\ell +1}\land \psi$ 
and $\omega$ is $\Sigma_{\ell +1}$-definable in $M$.
Since $(\omega,\mathrm{SSy}(M)) \models \RT^n_2$,
we get $(M,\Delta_{\ell}$-$\mathrm{Def}(M)) \models \RT^n_2$ by Theorem \ref{thm:rt-on-cut},
so $M \models \RR^n \land \psi$.

The proof that the $\Pi_3$ consequences of $\RR^n$ follow
from $\bd\Sigma_{1} + \exp$ is very similar, except that the function
$f$ is now defined by
\begin{align*}
f(y) & =\textrm{least } w > 2^y \textrm{ such that } \ab{y'}{y}\eb{z}{w}\pi(a,y',z),
\end{align*}
where $\pi$ is now a $\Delta_0$ formula. The difference is due to the fact that
for $\ell = 0$ we no longer have to care about elementarity between the cut $M$ and the model $K$
to ensure that $M \models \bd\Sigma_{\ell +1}\land \psi$, but we need to guarantee that $M \models \exp$.

We have thus proved \ref{it:rt32-fo-levels}. Regarding \ref{it:rt32-fo},
note that the containments \[\IB + \exp \subseteq \RR^n \subsetneq \PA\]
follow directly from the statement of \ref{it:rt32-fo-levels}, 
and in the proof of \ref{it:rt32-fo-levels} we constructed a model
of $\IB + \exp$ not satisfying $\RR^n$. 
Finally, observe that $\IB$ is not contained in any $\ind\Sigma_\ell$, 
so any subtheory of $\PA$ extending $\IB$
cannot be finitely axiomatizable.
\end{proof}

Note that the proof of Theorem \ref{thm:rt32-fragments-pa} immediately gives the following statement, 
which says essentially that Lemma \ref{lem:jockusch} is optimal with respect to the amount of induction used
to prove the existence of colourings without simple homogeneous sets.

\begin{corollary}
For each $\ell \ge 1, n\ge 2$, the theory $\bd\Sigma_{\ell} + \exp + \Delta_{\ell}$-$\RT^n_2$ is consistent.
\end{corollary}

\begin{remark}
As mentioned in Section \ref{sec:prelim}, results such as Theorem~\ref{thm:rt32-fragments-pa}
can be converted from purely arithmetical to $\Pi^{1}_{1}$ form by relativizing to second-order parameters.
In Theorem~\ref{thm:rt32-fragments-pa}\ref{it:rt32-fo}, the appropriate relativization of the scheme $\IB$ 
takes the form $\forall X\,(\ind \Sigma_k(X) \Rightarrow \bd \Sigma_{k+1}(X))$ for each $k$.
In Section \ref{sec:rt22}, we will also consider a weaker relativization of $\IB$: see the remark after Corollary~\ref{cor:rt22-cs2}.
\end{remark}

\begin{question}
Does $\RCA^*_0 + \RT^3_2$ imply $\RT^4_2$? 
More generally, does $\RCA^*_0 + \RT^n_2$ imply $\RT^{n+1}_2$ for some/all $n \ge 3$? 
\end{question}

\section{Ramsey for pairs}\label{sec:rt22}

We turn to the case of Ramsey's Theorem for pairs. Here, we are not able to give
a complete axiomatization analogous to that of Theorem \ref{thm:rt32-axioms}. 
Loosely speaking, our understanding of the strength of $\RCA^*_0 + \RT^2_2$ strongly depends on the amount
of induction satisfied by the underlying first-order model.

\begin{theorem}\label{thm:rt22-axioms} 
Let $\RR^2$ stand for the first-order consequences of $\RCA^*_0 + \RT^2_2$.
Then:
\begin{enumerate}[(a)]
\item\label{it:rt22-neg-is1} $\RR^2 \land \neg \ind \Sigma_1$ is axiomatized by 
$\bd\Sigma_1+\exp + \Delta_1$-$\RT^2_2$.
\item \label{it:rt22-is2} $\ind \Sigma_2$ implies $\RR^2$.
\item\label{it:rt22-bs2} Over $\bd\Sigma_2$, $\RR^2$ is implied by, and consistent with, 
both the first-order consequences of $\RCA_0 +\RT^2_2$ and the statement $\Delta_2$-$\RT^2_2$.
\item\label{it:rt22-is1} $\RR^2$ implies every first-order sentence $\psi$ such that both
$\bd \Sigma_2 \vdash \psi$ and $\RCA^*_0 + \neg \ind\Sigma^0_1 \vdash \psi$.
\end{enumerate}
\end{theorem}
\begin{proof}
We first prove \ref{it:rt22-neg-is1}. 
Clearly, if $M\models \bd\Sigma_1 + \exp$ and $(M,\Delta_{1}$-$\mathrm{Def}(M)) \models \RT^2_2$,
then $M$ satisfies $\RR^2$ (as well as $\neg\ind\Sigma_1$, by Lemma \ref{lem:jockusch}).
On the other hand, let $(M,\X) \models {\RCA^*_0} + {\RT^2_2} + {\neg\ind\Sigma_1}$.
Obviously, $M$ satisfies $\bd \Sigma_1 + \exp$. 
Let $I$ be a proper $\Sigma_1$-definable cut in $M$. 
Applying Theorem \ref{thm:rt-on-cut} two times, we get first $(I,\Cod(M/I)) \models \RT^2_2$
and then $(M,\Delta_{1}$-$\mathrm{Def}(M)) \models \RT^2_2$.

Statement \ref{it:rt22-is2} follows immediately from the result of \cite{cholak-jockusch-slaman}
that $\RCA_0 + \ind\Sigma^0_2 + \RT^2_2$ is conservative over $\ind\Sigma_2$.

We turn to \ref{it:rt22-bs2}. It is clear that $\RR^2$ is implied
by the first-order consequences of $\RCA_0 + \RT^2_2$.
Meanwhile, $\RR^2$ is also satisfied by any model $M \models \bd\Sigma_2 + \Delta_{2}$-$\RT^2_2$ since 
$(M, \Delta_2$-$\mathrm{Def}(M))\models \RCA^*_0+\RT^{2}_{2}$.
It remains to argue that such a model exists.
To see this, take $M \models \bd \Sigma_2$ with
$\Sigma_2$-definable $\omega$ and $(\omega,\mathrm{SSy}(M)) \models \RT^2_2$,
and apply Theorem \ref{thm:rt-on-cut} to the model
$(M, \Delta_2$-$\mathrm{Def}(M))\models \RCA^*_0$.

Finally, to see that \ref{it:rt22-is1} holds, 
let $\psi$ be provable both in $\bd \Sigma_2$ and in $\RCA^*_0 + \neg \ind\Sigma^0_1$.
We check that $\RCA^*_0 + \RT^2_2\vdash \psi$.
Let $(M,\X) \models \RCA^*_0 + \RT^2_2 $. 
If $(M,\X) \models \RCA_0$, then $M \models \bd\Sigma_2$,
so $M \models \psi$. 
Otherwise, $M \models \RCA^*_0 + \neg \ind\Sigma^0_1$,
so $M \models \psi$ as well.
\end{proof}

Parts \ref{it:rt22-neg-is1} and \ref{it:rt22-is2} of Theorem \ref{thm:rt22-axioms} give a complete axiomatization
of the first-order consequences of $\RCA^*_0 + \RT^2_2$ over, respectively, $\neg \ind\Sigma_1$ and $\ind \Sigma_2$.
However, the situation in the region between $\ind \Sigma_1$ and $\ind \Sigma_2$ is much less clear.

As mentioned in the introduction, it is open whether $\RCA_0 + \RT^2_2$ is arithmetically conservative over $\bd\Sigma_2$.
Therefore, it is consistent with what we know that already $\bd\Sigma_2$ implies 
the first-order consequences of $\RCA^*_0 + \RT^2_2$. 

On the other hand, we will now use
Theorem \ref{thm:rt22-axioms}\ref{it:rt22-is1} to show that there are some first-order sentences 
provable in $\RCA^*_0 + \RT^2_2$ but not in $\ind \Sigma_1$. It will be clear from our argument
that this is not a feature of $\RT^2_2$ specifically, but rather of all principles
that imply $\bd\Sigma^0_2$ (or even somewhat weaker statements) over $\RCA_0$.

\begin{definition}
For each $\ell \ge 1$, the $\Sigma_\ell$ \emph{cardinality scheme}, 
$\mathrm{C}\Sigma_\ell$, asserts that no $\Sigma_\ell$ formula
defines a total injection with bounded range.

The $\Sigma_\ell$ \emph{generalized pigeonhole principle}, $\mathrm{GPHP}(\Sigma_\ell)$,
asserts that for every $\Sigma_\ell$ formula $\varphi(x,y,z)$ and every number $a$,
there exists a number $b$ such that there is no $c$ for which $\varphi(\cdot,\cdot,c)$ defines
an injective multifunction from $b$ into $a$:
\[\forall a\, \exists b \, \forall c\,  [\as{x}{b}\es{y}{a}\varphi(x,y,c) \Rightarrow
\neg \as{y}{a}\exists^{\le 1}x\! < \! b \,\varphi(x,y,c)]. \]
\end{definition}

The principle $\mathrm{C}\Sigma_\ell$ was defined in \cite{seetapun-slaman}.
It is known that $\ind\Sigma_\ell$ does not imply $\mathrm{C}\Sigma_{\ell+1}$ \cite[Proposition 3.1]{groszek-slaman}. 
The principle $\mathrm{GPHP}(\Sigma_\ell)$ was defined in \cite{kaye:constructing-kappa-like},
where it was also observed that the theory of all $\kappa$-like models of $\ind\Delta_0$
implies $\mathrm{GPHP}(\Sigma_\ell)$ for all $\ell$.

Clearly, $\mathrm{GPHP}(\Sigma_\ell)$ implies $\mathrm{C}\Sigma_\ell$ for each $\ell \ge 1$.
For $\ell \ge 2$, $\mathrm{GPHP}(\Sigma_\ell)$ is in turn implied by $\bd\Sigma_\ell$,
since the latter is, for each $\ell \ge 1$, equivalent to the usual pigeonhole principle for $\Sigma_\ell$ maps
over $\ind\Delta_0 +\exp$ \cite{dimitracopoulos-paris:php}.
It follows from \cite{bcwwy:php-konig} that the implication from $\bd\Sigma_\ell$
to $\mathrm{GPHP}(\Sigma_\ell)$ is strict.

$\mathrm{C}\Sigma_2$ is known to be a consequence of some theories studied
in reverse mathematics that do not imply $\bd\Sigma_2$, 
such as~$\RCA_0$ plus the Rainbow Ramsey Theorem for pairs \cite{conidis-slaman} 
and $\RCA_0$ plus the existence of 2-random reals \cite{haken:thesis}.

In the theorem below, we explicitly indicate second-order variables to emphasize the role played by set parameters in the second part of the statement. Recall that $\ind\Sigma^{0}_{k}$ (resp.~$\bd\Sigma^{0}_{k}$) means $\forall X\,\ind\Sigma_{k}(X)$ (resp.~$\forall X\,\bd\Sigma_{k}(X)$).

\begin{theorem}\label{thm:neg-is1-cs2}
For each $k,\ell \ge 1$, the following statements are provable in $\RCA^*_0$:
\begin{enumerate}[(a)]
\item\label{it:bs2-cs2} $\forall X\,(\bd\Sigma_\ell(X)\Rightarrow \mathrm{GPHP}(\Sigma_\ell(X)))$,  
\item\label{it:neg-is1-cs2} $(\bd\Sigma^{0}_{k} \wedge\neg\ind\Sigma^0_{k})\Rightarrow \forall X\,\mathrm{GPHP}(\Sigma_\ell(X))$.
\end{enumerate}
\end{theorem}

Theorem \ref{thm:neg-is1-cs2} part \ref{it:neg-is1-cs2} can be obtained by relativizing Kaye's proof of the result  that any model of $\bd\Sigma_1 + \exp + {\neg\ind\Sigma_1}$ is elementarily equivalent to an $\aleph_\omega$-like structure \cite[Theorem 2.4]{kaye:constructing-kappa-like}. 
A model of $\neg\ind\Sigma_1(A) + \neg \mathrm{GPHP}(\Sigma_\ell(B)) + \bd\Sigma_1(A\oplus B) + \exp$ would also be elementary equivalent to $\aleph_\omega$-like model, but clearly such a structure can never violate the scheme 
$\mathrm{GPHP}(\Gamma)$ for any class of formulas $\Gamma$.

The proof of Theorem \ref{thm:neg-is1-cs2} we give below is considerably simpler than that of
\cite[Theorem 2.4]{kaye:constructing-kappa-like}. On the other hand, both make use of an automorphism argument.
It would be interesting to come up with a direct proof of $\mathrm{GPHP}(\Sigma_\ell)$, with no model-theoretic detours,
in for instance ${\bd\Sigma_1} + {\exp} + {\neg\ind\Sigma_1}$, .

\begin{proof}
It has already been mentioned that $\bd\Sigma_\ell + \exp$ implies $\mathrm{GPHP}(\Sigma_\ell)$. 
The argument for this relativizes with no issues, thus proving part \ref{it:bs2-cs2}.

It remains to prove that $\RCA^*_0 +\bd\Sigma^{0}_{k} + \neg\ind\Sigma^0_{k}$ 
implies $\mathrm{GPHP}(\Sigma^0_\ell)$ for any $\ell$. To simplify notation, we restrict
ourselves to the case where $k=1$ and to $\mathrm{GPHP}$ for lightface $\Sigma_\ell$ formulas.
The general case for $k \ge 1$ and a $\Sigma_\ell(B)$ formula
reduces to this one by considering the model of $\RCA^*_0$ given by the 
$\Delta_k(A\oplus B)$-definable sets, where $A$ is a parameter witnessing the failure of $\ind\Sigma^0_k$.

Let $(M,A)$ be a countable model of $\bd\Sigma_1(A) + {\exp} + \neg\ind\Sigma_1(A)$.
We may assume that $A$ itself has an increasing enumeration $A = \{a_i: i \in I\}$
for a proper cut $I \subseteq M$.
By a routine compactness argument, we may also assume that for every $a \in M$
there is some $b \in M$ such that $b > \exp_m(a)$ for each $m \in \omega$.
To prove that $M \models \mathrm{GPHP}(\Sigma_\ell)$, we will use a technique based
on the fact that models of $\bd\Sigma^0_1 + {\exp} + {\neg\ind\Sigma^0_1}$ 
have many automorphisms \cite{kossak:extensions, kossak:extensions-correction, kaye:properties}. 

By a standard argument (see e.g.~\cite[Theorem 4.6]{enayat-wong}), the model $M$ can be end-extended
to a model $K \models \ind \Delta_0$ such that $A \in \Cod(K/M)$. Since elements
coding $A$ are downwards cofinal in $K \setminus M$, there is an element $d \in K$
coding $A$ and small enough that $\exp_2(d)$ exists in $K$. By \cite{lessan:thesis},
there is a $\Delta_0$ formula with parameter $\exp_2(d)$ that
defines satisfaction for $\Delta_0$ formulas on arguments below $d$.
As a consequence, the structure $[0,d]$ (with addition and multiplication as ternary relations) 
is recursively saturated.

Now let $a \in M\setminus I$ and let $b\in M$ be such that $b > \exp_m(a)$ for each $m \in \omega$.
Let $c \in M$ be arbitrary. 
The recursive saturation of $[0,d]$ lets us use an argument dating back to 
\cite{kotlarski:elem-cuts-rec-sat} (see the proof of Lemma 3.4 in \cite{kaye:properties} 
for a detailed argument and \cite{kossak:extensions-correction} for a brief discussion)
to derive the existence of an automorphism $\alpha$ of $[0,d]$ such that
$\alpha$ fixes $c,d$ and fixes $[0,a]$ pointwise, but there is some $x < b$ with $x \neq \alpha(x)=:y$. 
For each $i \in I$, since $\alpha(i) = i$,
$\alpha(d) = d$, and $d$ codes $A$, we know that $\alpha(a_i) = a_i$.
Therefore, $\alpha[M] = M$, so $\alpha{\upharpoonright}_M$ is actually an automorphism of $M$.
We now argue that no injective multifunction from $b$ to $a$ is definable in $M$ with $c$ as parameter.
Otherwise, if $f$ were such a multifunction, there would be some $z < a$ such that
$z \in f(x)$, and therefore (since $\alpha$ fixes both $z$ and $c$) also
$z = \alpha(z) \in f(\alpha(x)) = f(y)$. By the injectivity of $f$, this
would imply $x = y$, a contradiction. Since $c \in M$ was arbitrary,
this proves that there can be no injective multifunction from $b$ to $a$
definable in $M$, so $M \models \mathrm{GPHP}(\Sigma_\ell)$ for each $\ell$.
\end{proof}

\begin{corollary}\label{cor:rt22-cs2}
$\RCA^*_0 + \RT^2_2$ proves both $\mathrm{C}\Sigma_2$ and $\mathrm{GPHP}(\Sigma_2)$.
\end{corollary}
\begin{proof}
This is a direct consequence of Theorem \ref{thm:rt22-axioms}\ref{it:rt22-is1}, Theorem \ref{thm:neg-is1-cs2}\ref{it:neg-is1-cs2},
and the fact that $\mathrm{GPHP}(\Sigma_\ell)$ implies $\mathrm{C}\Sigma_\ell$.
\end{proof}


\begin{remark}
Let the usual relativization of $\IB$, namely $\forall X\,(\ind \Sigma_k(X) \Rightarrow \bd \Sigma_{k+1}(X))$ for each $k$,
be called ``strong'', and let the ``weak'' relativization of $\IB$ consist of the statements 
$\ind\Sigma^0_k \Rightarrow \bd \Sigma^0_{k+1}$ for each $k$. 
In Theorem \ref{thm:rt32-fragments-pa}, we showed that $\RCA^*_0 + \RT^3_2$ implies strong relativized $\IB$. 
On the other hand, Theorem \ref{thm:neg-is1-cs2} implies that already weak relativized $\IB$, 
and even its restriction to $k < \ell$, suffices to prove $\mathrm{GPHP}(\Sigma_\ell)$.
 
This lets us prove Corollary \ref{cor:rt22-cs2} by exploiting the fact that $\RCA^*_0 + \RT^2_2$ 
implies the restriction of weak relativized $\IB$ to $k = 0,1$.
\end{remark}

The known relationships between the first-order consequences of $\RCA^*_0 + \RT^2_2$
and fragments of first-order arithmetic are summarized in the following corollary. 

\begin{corollary}\label{cor:rt22-fragments-pa}
The first-order consequences of $\RCA^*_0 + \RT^2_2$ follow from $\ind \Sigma_2$.
The $\Pi_3$ consequences coincide with $\bd \Sigma_1 + \exp$.
The $\Pi_4$ consequences are strictly weaker than $\bd \Sigma_2$ but do not follow from $\ind \Sigma_1$.
\end{corollary}
\begin{proof}
The provability from $\ind\Sigma_2$ is part \ref{it:rt22-is2} of Theorem \ref{thm:rt22-axioms}.
The fact that the $\Pi_3$ consequences of $\RCA^*_0 + \RT^2_2$ coincide with $\bd \Sigma_1 + \exp$
and that the $\Pi_4$ consequences are strictly weaker than $\bd \Sigma_2$ is proved like in Theorem \ref{thm:rt32-fragments-pa}.
Finally, Corollary \ref{cor:rt22-cs2} implies that $\mathrm{C}\Sigma_2$ is an example
of a $\Pi_4$ sentence that follows from $\RCA^*_0 + \RT^2_2$ but not $\ind\Sigma_1$.
\end{proof}

Of course, quite a few questions remain. Over $\bd\Sigma_2$, 
one basic issue is whether the first-order consequences
of $\RCA^*_0 + \RT^2_2 + \bd\Sigma_2$ are non-trivial, and another is 
how closely related they are to those of $\RCA_0 + \RT^2_2$.

\begin{question}
Is $\RCA^*_0 + \RT^2_2 + \bd\Sigma_2$ conservative over $\bd\Sigma_2$?
\end{question}

\begin{question}
Does $\RCA^*_0 + \RT^2_2 + \bd\Sigma_2$ imply
$\psi  \lor (\bd \Sigma_2 \land \Delta_2\textrm{-}\RT^2_2)$ 
for each first-order $\psi$ provable in $\RCA_0 + \RT^2_2$?  
\end{question}

Over $\ind\Sigma_1$, the basic question is:
\begin{question}
Does $\RCA^*_0 +\RT^2_2  + \ind\Sigma_1$ imply $\bd\Sigma_2$? 
\end{question}

We have no strong reasons to believe that the answer is ``yes''.
However, it should be pointed out that, since $\RCA_0 + \RT^2_2$ proves $\bd\Sigma_2$,
answering ``no'' would involve constructing a model of $\ind\Sigma_1+\neg\bd\Sigma_2$
that expands to a model of $\bd\Sigma^0_1 + \neg\ind\Sigma^0_1$ -- 
in the terminology of \cite{kossak:extensions}, a model of $\ind\Sigma_1+\neg\bd\Sigma_2$
that is not always semiregular. The existence of such a model itself seems to be open.

\begin{question}
Does there exist a model $M \models \ind\Sigma_1 + \neg \bd\Sigma_2$ 
that can be expanded to a model $(M,A) \models \bd\Sigma_1(A) + \neg \ind\Sigma_1(A)$?
\end{question}

Note that if there is $M$ witnessing a positive answer to this question such that
$\ind\Sigma_1(A)$ fails in the expansion due to $\omega$ being $\Sigma_1(A)$-definable, 
then by Theorems \ref{thm:rt-on-cut} and \ref{thm:rt32-fragments-pa}
it has to be the case that $(\omega,\mathrm{SSy}(M)) \not \models \ACA_0$.

\section{Relativizing Ramsey}\label{sec:delta02-rt22}

In this final section, we take up the question whether our results on 
$\RCA^*_0 + \RT^2_2$ shed any light on the problem of characterizing the first-order
consequences of $\RCA_0 +\RT^2_2$. To this end, we introduce a principle
in which both the instances and solutions to Ramsey's Theorem are allowed
to be $\Delta^0_2$-sets rather than sets.

\begin{definition}
$\Delta^0_2$-$\RT^2_2$ is the $\Pi^1_2$ statement: ``for every $\Delta^0_2$-set $f$ which is a $2$-colouring of
$[\N]^2$, there exists an infinite homogeneous $\Delta^0_2$-set''.
\end{definition}

Note that $\Delta^0_2$-$\RT^2_2$ is a genuine $\Pi^1_2$ statement, which
should not be confused with the $\Pi^1_1$ statement relativizing
$\Delta_2$-$\RT^2_2$, namely ``for every set $X$, $\Delta_2(X)$-$\RT^2_2$ holds''.
Of course, in a model of the form $(M, \Delta_1\textrm{-}\mathrm{Def}(M))$,
the statement $\Delta^0_2$-$\RT^2_2$ will be equivalent to $\Delta_2$-$\RT^2_2$.

We are interested in studying $\Delta^0_2$-$\RT^2_2$ over $\RCA_0 + \bd\Sigma^0_2$, especially in the case where $\ind \Sigma^0_2$ fails. The following proposition shows that in such a context, $\Delta^0_2$-$\mathrm{RT}^2_2$
behaves somewhat analogously to $\RT^2_2$ over $\RCA^*_0 +\RT^2_2$, so we can investigate it
using the methods developed in Sections \ref{sec:cuts}-\ref{sec:rt22}. 

\begin{lemma}\label{lem:delta02-rt22}
For any model $(M,\X) \models \RCA_0 + \bd\Sigma^0_2$: 
$(M,\X) \models \Delta^0_2$-$\mathrm{RT}^2_2$ iff $(M, \Delta^0_2\textrm{-}\mathrm{Def}(M,\X)) \models \RCA^*_0 + \RT^2_2$. As a consequence:
\begin{enumerate}[(a)]
\item\label{it:delta02-rt22-on-cut} if $I$ is a $\Sigma^0_2$-definable proper cut in $(M,\X)$, then $(M,\X) \models \Delta^0_2$-$\mathrm{RT}^2_2$
iff $(I,\Cod(M/I)\models \RT^2_2$,
\item\label{it:delta02-rt22-axioms} the first-order consequences of ${\RCA_0} + {\bd\Sigma_2} + {\neg \ind\Sigma_2} + {\Delta^0_2}$-$\mathrm{RT}^2_2$ are axiomatized by $\bd\Sigma_2 + \Delta_2$-$\mathrm{RT}^2_2$,
\item\label{it:delta02-rt22-conservativity} $\RCA_0 + \bd\Sigma^0_2 + \Delta^0_2$-$\mathrm{RT}^2_2$ is $\Pi_4$- but not $\Pi_5$-conservative
over $\bd\Sigma_2$.
\end{enumerate}
\end{lemma}
\begin{proof}
The fact that a model $(M,\X)$ satisfies $\RCA_0 + \bd\Sigma^0_2 + \Delta^0_2$-$\mathrm{RT}^2_2$ exactly if $(M, \Delta^0_2\textrm{-}\mathrm{Def}(M,\X)) \models \RCA^*_0 + \RT^2_2$ is immediate from the definitions. Thus
\ref{it:delta02-rt22-on-cut} follows from Theorem \ref{thm:rt-on-cut}, because a cut $I$ is $\Sigma^0_2$ definable
in $(M,\X) \models \bd\Sigma^0_2$ exactly if it is $\Sigma^0_1$-definable in $(M, \Delta^0_2\textrm{-}\mathrm{Def}(M,\X))$.

To prove \ref{it:delta02-rt22-axioms}, repeat the argument 
from the proof of Theorem \ref{thm:rt22-axioms}\ref{it:rt22-neg-is1},
relativizing it to $0'$. If $M \models \bd \Sigma_2 + \Delta_2$-$\mathrm{RT}^2_2$, then
$(M, \Delta_1\textrm{-}\mathrm{Def}(M)) \models {\RCA_0} + {\bd\Sigma_2} + {\neg \ind\Sigma_2} + {\Delta^0_2}$-$\mathrm{RT}^2_2$. In the other direction,
if $(M, \X) \models {\RCA_0} + {\bd\Sigma_2} + {\neg \ind\Sigma_2} + {\Delta^0_2}$-$\mathrm{RT}^2_2$
and $I$ is a proper $\Sigma_2$-definable cut in $M$, then two applications 
of \ref{it:delta02-rt22-on-cut} give first $(I,\Cod(M/I)\models \RT^2_2$ and then
$(M, \Delta_1\textrm{-}\mathrm{Def}(M)) \models \Delta^0_2$-$\mathrm{RT}^2_2$, but the latter
is equivalent to $M \models \Delta_2$-$\mathrm{RT}^2_2$.

To show that ${\RCA_0} + {\bd\Sigma^0_2} + {\Delta^0_2}$-$\mathrm{RT}^2_2$ is $\Pi_4$-conservative
over $\bd\Sigma_2$, relativize to $0'$ the argument used to prove $\Pi_3$-conservativity of $\RCA^*_0 + \RT^n_2$
over $\bd\Sigma_1+ \exp$ in Theorem \ref{thm:rt32-fragments-pa}\ref{it:rt32-fo-levels}. 
To show lack of $\Pi_5$-conservativity, consider the sentence $\neg\ind\Sigma_2 \Rightarrow \Delta_2$-$\mathrm{RT}^2_2$.
This is a $\Pi_5$ statement, and it is provable in ${\RCA_0} + {\bd\Sigma^0_2} + {\Delta^0_2}$-$\mathrm{RT}^2_2$ 
by \ref{it:delta02-rt22-axioms}.
On the other hand, it is not provable in $\bd\Sigma_2$, as can be seen
by applying \ref{it:delta02-rt22-on-cut} to any model $M \models \bd \Sigma_2$ with $\Sigma_2$-definable $\omega$ and
$(\omega, \mathrm{SSy}(M)) \not \models \RT^2_2$. This proves \ref{it:delta02-rt22-conservativity}.
\end{proof}

Since Lemma \ref{lem:delta02-rt22} shows that $\Delta^0_2$-$\mathrm{RT}^2_2$ is not $\Pi_5$-conservative over 
$\bd\Sigma_2$,
while the conservativity of $\RCA_0 + \RT^2_2$  over $\bd\Sigma_2$ is a well-known open problem, 
it is natural to ask whether $\RT^2_2$ might imply $\Delta^0_2$-$\mathrm{RT}^2_2$, at least
in the particularly relevant setting of models of $\bd\Sigma^0_2 + \neg \ind \Sigma^0_2$.

In Theorem \ref{thm:rt22-vs-delta02-rt22} below, we show a negative result: there is no implication in either direction, 
and the sentence we used to prove lack of $\Pi_5$-conservativity of $\Delta^0_2$-$\mathrm{RT}^2_2$ is unprovable in $\RT^2_2$.
To prove this, we will have to make use of a connection between properties of infinite $\Delta_2$-sets 
and the consistency of $\ind \Sigma_1$ that may probably be considered folklore, 
but for which we did not find a suitable reference.
So, we state the connection as a separate lemma and sketch its proof in Appendix A.

\begin{lemma}\label{lem:d2-set-con-is1}
There exists a polynomial $p$ such that $\ind\Sigma_1$ proves:
\begin{multline*}
\forall x\textrm{ \big{[}``every infinite } \Delta_2 \textrm{-set contains at least } {\exp_{p(x)}(2)} \textrm{ elements''}
\\ \Rightarrow \Con_x(\ind \Sigma_1) \big{]},
\end{multline*}
where $\Con_{x}(T)$ means that there is no inconsistency proof in $T$ containing fewer than $x$ symbols.
\end{lemma}

It may be worth pointing out that Lemma \ref{lem:d2-set-con-is1}
is a quantitative version of a weakening of the well-known fact that $\ind\Sigma_2$ 
is equivalent to uniform $\Pi_4$-reflection for $\ind \Delta_0 + \exp$ (see e.g.~\cite[Theorem 7]{beklemishev:reflection-survey}).
To see this, note that (over $\ind\Delta_0 + \exp$ as a base theory) 
$\ind \Sigma_2$ is equivalent to the statement that 
each infinite $\Delta_2$-set contains arbitrarily large finite sets,
while $\Pi_4$-reflection for $\ind \Delta_0 + \exp$ implies $\Con(\ind\Sigma_1)$.

\begin{theorem}\label{thm:rt22-vs-delta02-rt22}
$\RT^2_2$ and $\Delta^0_2$-$\mathrm{RT}^2_2$ are incomparable over $\RCA_0 + \bd\Sigma^0_2 + \neg \ind \Sigma^0_2$. Moreover, $\RCA_0 + \RT^2_2$ does not prove $\neg \ind\Sigma_2 \Rightarrow \Delta_2$-$\mathrm{RT}^2_2$.
\end{theorem}
\begin{proof}
The fact that ${\RCA_0} + {\bd\Sigma^0_2} + {\neg \ind \Sigma^0_2} + {\Delta^0_2}$-$\mathrm{RT}^2_2$ does not prove 
$\RT^2_2$ is witnessed by any structure of the form $(M, \Delta_1\textrm{-}\mathrm{Def}(M))$,
where $M \models \bd \Sigma_2$ has $\Sigma_2$-definable $\omega$ and $(\omega,\mathrm{SSy}(M)) \models \RT^2_2$.
By Lemma \ref{lem:delta02-rt22}\ref{it:delta02-rt22-on-cut}, 
such a structure satisfies $\Delta^0_2$-$\mathrm{RT}^2_2$,
but by Lemma \ref{lem:jockusch}\ref{it:jockusch-1} it cannot satisfy $\RT^2_2$.

In the other direction, such a ``quick and dirty'' argument does not seem to be currently available:
of the known constructions producing models of ${\RCA_0} + {\RT^2_2} + {\neg \ind\Sigma^0_2}$,
that of \cite{csy:meta-stable-ramsey, csy:ind-strength-ramsey} 
involves strong constraints on $\mathrm{SSy}(M)$, and that of \cite{py:rt22, ky:ordinal-valued-ramsey} 
does not give a $\Sigma^0_2$-definable $\omega$.
To show that ${\RCA_0} + {\RT^2_2} + {\neg \ind \Sigma^0_2}$ does not imply
$\Delta^0_2$-$\mathrm{RT}^2_2$, it is enough to prove the ``Moreover'' part of the statement, namely:
\[ \RCA_0 + \RT^2_2 \not \vdash \neg \ind\Sigma_2 \Rightarrow \Delta_2\textrm{-}\mathrm{RT}^2_2.\]
This we do by means of a proof speedup argument. 
By \cite[Lemma 3.2]{kwy:ramsey-proof-size}, $\RCA^*_0 + \RT^2_2$
proves the statement ``for every $k$, if every infinite set contains at least $k$ elements, 
then every infinite set contains at least $2^k$ elements''.
It follows immediately that
$\bd \Sigma_2 + \Delta_2\textrm{-}\mathrm{RT}^2_2$ proves 
``for every $k$, if every infinite $\Delta_2$-set contains at least $k$ elements, 
then every infinite $\Delta_2$-set contains at least $2^k$ elements''.
This implies that the definable set 
\[\{x: \textrm{every infinite } \Delta_2 \textrm{-set contains at least } \exp_x(2) \textrm{ elements }\}\]
is a cut in $\bd \Sigma_2 + \Delta_2\textrm{-}\mathrm{RT}^2_2$. 
This in turn implies (cf.~\cite[Theorem 3.4.1]{pudlak:handbook-lengths-of-proofs}) that, 
for each $n \in \omega$, there is a $\mathrm{poly}(n)$-size proof of 
\[\textrm{``every infinite } \Delta_2 \textrm{-set contains at least } \exp_{\exp_n(2)}2 \textrm{ elements''}\]
in $\bd \Sigma_2 + \Delta_2\textrm{-}\mathrm{RT}^2_2$. 
But by Lemma \ref{lem:d2-set-con-is1} and the fact that the exponential function dominates
every polynomial, $\ind \Sigma_1$ proves:
\begin{multline*}
\forall x\textrm{ \big{[}``every infinite } \Delta_2 \textrm{-set contains at least } \exp_{\exp_{x+1}(2)}2 \textrm{ elements''}
\\ \Rightarrow \Con_{\exp_x(2)}(\ind \Sigma_1) \big{]}.
\end{multline*}
Thus, for each standard $n$ there is a 
$\mathrm{poly}(n)$-size proof of $\Con_{\exp_n(2)}(\ind \Sigma_1)$ in 
$\bd \Sigma_2 + \Delta_2\textrm{-}\mathrm{RT}^2_2$. 

Reasoning by cases, we can show that also  
${\bd \Sigma_2} + {(\neg \ind\Sigma_2 \Rightarrow \Delta_2\textrm{-}\mathrm{RT}^2_2)}$ proves 
$\Con_{\exp_n(2)}(\ind \Sigma_1)$ in $\mathrm{poly}(n)$-size.
Indeed, either $\ind \Sigma_2$ holds, in which case we simply have $\Con(\ind \Sigma_1)$, 
or $\ind \Sigma_2$ fails, in which case we have $\Delta_2\textrm{-}\mathrm{RT}^2_2$ 
and we can use the proof of $\Con_{\exp_n(2)}(\ind \Sigma_1)$ mentioned in the previous paragraph.

However, the size of the smallest proof of $\Con_{\exp_n(2)}(\ind \Sigma_1)$ in $\ind \Sigma_1$ 
grows nonelementarily in $n$ \cite[Theorem 7.2.2]{pudlak:handbook-lengths-of-proofs}, and by \cite{kwy:ramsey-proof-size},
$\RCA_0 + \RT^2_2$ has no superpolynomial proof speedup over $\ind \Sigma_1$
w.r.t.~proofs of $\Pi_3$ sentences. 
Thus, the size of the smallest proof of $\Con_{\exp_n(2)}(\ind \Sigma_1)$ in $\RCA_0 + \RT^2_2$
also grows nonelementarily in $n$. 
Since $\bd \Sigma_2 + (\neg \ind\Sigma_2 \Rightarrow \Delta_2\textrm{-}\mathrm{RT}^2_2)$
is axiomatized by a single sentence, and $\RCA_0 + \RT^2_2$ proves $\bd \Sigma_2$,
it follows that it cannot prove $\neg \ind\Sigma_2 \Rightarrow \Delta_2\textrm{-}\mathrm{RT}^2_2$.
\end{proof}

Thus, the statement $\neg \ind\Sigma_2 \Rightarrow \Delta_2\textrm{-}\mathrm{RT}^2_2$ cannot be used to 
witness the potential nonconservativity of $\RT^2_2$ over $\bd\Sigma_2$. 
However, our argument for this, in addition to being somewhat roundabout, made use of the fact that
$\RCA^*_0 + \RT^2_2$ proves ``for every $k$, if every infinite set contains at least $k$ elements, 
then every infinite set contains at least $2^k$ elements'', which is shown using exponential
lower bounds on finite Ramsey numbers. 
Thus the argument is no longer applicable
to various apparently slight weakenings of $\neg \ind\Sigma_2 \Rightarrow \Delta_2\textrm{-}\mathrm{RT}^2_2$,
for instance to statements in which $\RT^2_2$ is replaced by a restriction to colourings 
for which finite Ramsey numbers are polynomial.

As an illustration, we mention two weakenings of $\neg \ind\Sigma_2 \Rightarrow \Delta_2\textrm{-}\mathrm{RT}^2_2$
whose status is open and seems intriguing.

\begin{question}
Does $\RCA_0 + \RT^2_2$ prove one of the following the $\Pi_5$ statements:
\begin{enumerate}[(a)]
\item $\neg \ind\Sigma_2 \Rightarrow \Delta_2\textrm{-}\mathrm{CAC}$: if $\neg \ind \Sigma_2$, 
then every $\Delta_2$-definable partial order on $[\N]$ contains 
an infinite $\Delta_2$-definable chain 
or an infinite $\Delta_2$-definable antichain'',
\item ``if $\neg \ind \Sigma_2$, 
then for every $\Delta_1$-definable $2$-colouring of $[\N]^n$ there is a $\Delta_2$-definable infinite homogeneous set''?
\end{enumerate}
Does $\RCA_0 + \bd \Sigma^0_2$ prove the statement in (b)?
\end{question}

\paragraph{Acknowledgment.} We are very grateful to Tin Lok Wong for many discussions and for making some observations that inspired the work presented in this paper. In fact, we felt that Wong's contribution to the paper are such that he should be listed as a coauthor, but he does not.

We are also grateful to Marta Fiori Carones for a thorough reading of the draft and number of useful comments.

\bibliographystyle{plain}
\bibliography{leszek2014}

\begin{thebibliography}{10}

\bibitem{beklemishev:reflection-survey}
Lev~D. Beklemishev.
\newblock Reflection schemes and provability algebras in formal arithmetic.
\newblock {\em Russian Math. Surveys}, 60(2):197--268, 2005.

\bibitem{bcwwy:php-konig}
David Belanger, Chitat Chong, Wei Wang, Tin~Lok Wong, and Yue Yang.
\newblock Where pigeonhole principles meet {K}önig lemmas, 2019.
\newblock Preprint. Available at \url{arXiv:1912.03487}.

\bibitem{belanger:coh}
David~R. Belanger.
\newblock Conservation theorems for the cohesiveness principle.
\newblock Preprint, 2015.

\bibitem{cholak-jockusch-slaman}
Peter~A. Cholak, Carl~G. Jockusch, and Theodore~A. Slaman.
\newblock On the strength of {R}amsey's theorem for pairs.
\newblock {\em J. Symb. Log.}, 66(1):1--55, 2001.

\bibitem{chong-mourad:degree-cut}
C.~T. Chong and K.~J. Mourad.
\newblock The degree of a {$\Sigma_n$} cut.
\newblock {\em Ann. Pure Appl. Logic}, 48(3):227--235, 1990.

\bibitem{csy:meta-stable-ramsey}
C.~T. Chong, Theodore~A. Slaman, and Yue Yang.
\newblock The metamathematics of stable {R}amsey's theorem for pairs.
\newblock {\em J. Amer. Math. Soc.}, 27(3):863--892, 2014.

\bibitem{csy:ind-strength-ramsey}
C.~T. Chong, Theodore~A. Slaman, and Yue Yang.
\newblock The inductive strength of {R}amsey's {T}heorem for {P}airs.
\newblock {\em Adv. Math.}, 308:121--141, 2017.

\bibitem{conidis-slaman}
Chris~J. Conidis and Theodore~A. Slaman.
\newblock {Random reals, the rainbow {R}amsey theorem, and arithmetic
  conservation}.
\newblock {\em J. Symb. Log.}, 78(1):195--206, 2013.

\bibitem{dimitracopoulos-paris:php}
C.~Dimitracopoulos and J.~Paris.
\newblock The pigeonhole principle and fragments of arithmetic.
\newblock {\em Z. Math. Logik Grundlag. Math.}, 32(1):73--80, 1986.

\bibitem{enayat-wong}
Ali Enayat and Tin~Lok Wong.
\newblock Unifying the model theory of first-order and second-order arithmetic
  via {${\rm WKL}_0^*$}.
\newblock {\em Ann. Pure Appl. Logic}, 168(6):1247--1283, 2017.

\bibitem{groszek-slaman}
Marcia~J. Groszek and Theodore~A. Slaman.
\newblock On {T}uring reducibility.
\newblock Preprint, 1994.

\bibitem{haken:thesis}
Ian~Robert Haken.
\newblock {\em Randomizing Reals and the First-Order Consequences of Randoms}.
\newblock PhD thesis, UC Berkeley, 2014.

\bibitem{hatzikiriakou:algebraic-disguises}
Kostas Hatzikiriakou.
\newblock Algebraic disguises of {$\Sigma^0_1$} induction.
\newblock {\em Arch. Math. Logic}, 29(1):47--51, 1989.

\bibitem{hirschfeldt:slicing}
Denis~R. Hirschfeldt.
\newblock {\em Slicing the truth. On the computable and reverse mathematics of
  combinatorial principles}.
\newblock World Scientific, 2015.

\bibitem{jockusch:ramsey}
Carl~G. Jockusch.
\newblock Ramsey's theorem and recursion theory.
\newblock {\em J. Symb. Log.}, 37(2):268--280, 1972.

\bibitem{kaye:properties}
Richard Kaye.
\newblock Model-theoretic properties characterizing {P}eano arithmetic.
\newblock {\em J. Symb. Log.}, 56(3):949--963, 1991.

\bibitem{Kaye91}
Richard Kaye.
\newblock {\em Models of {P}eano {A}rithmetic}.
\newblock Oxford University Press, 1991.

\bibitem{kaye:constructing-kappa-like}
Richard Kaye.
\newblock Constructing {$\kappa$}-like models of arithmetic.
\newblock {\em J. London Math. Soc. (2)}, 55(1):1--10, 1997.

\bibitem{kwy:ramsey-proof-size}
Leszek~Aleksander Ko{\l}odziejczyk, Tin~Lok Wong, and Keita Yokoyama.
\newblock Ramsey's theorem for pairs, collection, and proof size, 2020.
\newblock Submitted. Available at \url{arXiv:2005.06854}.

\bibitem{ky:categorical}
Leszek~Aleksander Ko{\l}odziejczyk and Keita Yokoyama.
\newblock Categorical characterizations of the natural numbers require
  primitive recursion.
\newblock {\em Ann. Pure Appl. Logic}, 166(2):219--231, 2015.

\bibitem{ky:ordinal-valued-ramsey}
Leszek~Aleksander Ko{\l}odziejczyk and Keita Yokoyama.
\newblock Some upper bounds on ordinal-valued {R}amsey numbers for colourings
  of pairs.
\newblock {\em Selecta Math. (N.S.)}, 26(4):paper No. 56, 18 pages, 2020.

\bibitem{kossak:extensions}
Roman Kossak.
\newblock On extensions of models of strong fragments of arithmetic.
\newblock {\em Proc. Amer. Math. Soc.}, 108(1):223--232, 1990.

\bibitem{kossak:extensions-correction}
Roman Kossak.
\newblock A correction to: ``{O}n extensions of models of strong fragments of
  arithmetic'' [{P}roc. {A}mer. {M}ath. {S}oc. {\bf 108} (1990), no. 1,
  223--232; {MR}0984802 (90d:03123)].
\newblock {\em Proc. Amer. Math. Soc.}, 112(3):913--914, 1991.

\bibitem{kotlarski:elem-cuts-rec-sat}
Henryk Kotlarski.
\newblock On elementary cuts in recursively saturated models of {P}eano
  arithmetic.
\newblock {\em Fund. Math.}, 120(3):205--222, 1984.

\bibitem{lessan:thesis}
Hamid Lessan.
\newblock {\em Models of arithmetic}.
\newblock PhD thesis, University of Manchester, 1978.

\bibitem{py:rt22}
Ludovic Patey and Keita Yokoyama.
\newblock The proof-theoretic strength of {R}amsey's theorem for pairs and two
  colors.
\newblock {\em Adv. Math.}, 330:1034--1070, 2018.

\bibitem{pudlak:handbook-lengths-of-proofs}
Pavel Pudl\'ak.
\newblock The lengths of proofs.
\newblock In S.~R. Buss, editor, {\em Handbook of {P}roof {T}heory}, pages
  547--642. Elsevier, 1998.

\bibitem{seetapun-slaman}
David Seetapun and Theodore~A. Slaman.
\newblock On the strength of {R}amsey's theorem.
\newblock {\em Notre Dame J. Form. Log.}, 36(4):570--582, 1995.

\bibitem{simpson:sosoa}
Stephen~G. Simpson.
\newblock {\em Subsystems of {S}econd {O}rder {A}rithmetic}.
\newblock Association for Symbolic Logic, 2009.

\bibitem{simpson-smith}
Stephen~G. Simpson and Rick~L. Smith.
\newblock Factorization of polynomials and {$\Sigma^0_1$} induction.
\newblock {\em Ann. Pure Appl. Logic}, 31(2-3):289--306, 1986.

\bibitem{sy:peanocat}
Stephen~G. Simpson and Keita Yokoyama.
\newblock Reverse mathematics and {P}eano categoricity.
\newblock {\em Ann. Pure Appl. Logic}, 164(3):284--293, 2012.

\bibitem{specker:ramsey}
Ernst Specker.
\newblock Ramsey's theorem does not hold in recursive set theory.
\newblock In {\em Logic {C}olloquium '69 ({P}roc. {S}ummer {S}chool and
  {C}olloq., {M}anchester, 1969)}, pages 439--442. North-Holland, Amsterdam,
  1971.

\bibitem{yokoyama:rt-rca*0}
Keita Yokoyama.
\newblock On the strength of {R}amsey's theorem without {$\Sigma_1$}-induction.
\newblock {\em MLQ Math. Log. Q.}, 59(1-2):108--111, 2013.

\end{thebibliography}

\appendix

\section{Proof of Lemma \ref{lem:d2-set-con-is1}}

\begin{consistency-lemma}
There exists a polynomial $p$ such that $\ind\Sigma_1$ proves:
\begin{multline*}
\forall x\textrm{ \big{[}``every infinite } \Delta_2 \textrm{-set contains at least } {\exp_{p(x)}(2)} \textrm{ elements''}
\\ \Rightarrow \Con_x(\ind \Sigma_1) \big{]},
\end{multline*}
where $\Con_{x}(T)$ means that there is no inconsistency proof in $T$ containing fewer than $x$ symbols.
\end{consistency-lemma}

\begin{proof}
We assume that our proof system is a Tait-style calculus (see e.g.~\cite[Section 4.1]{beklemishev:reflection-survey}).
Thus, $\land, \lor, \neg$ are our only connectives,
with $\neg$ allowed to appear explicitly only in front of atoms
and negation otherwise defined recursively using the De Morgan laws.
The proof lines are \emph{cedents}, or finite sets of formulas interpreted as disjunctions.
The logical axioms are cedents of the form $\Gamma, \psi, \neg \psi$ for $\psi$ atomic, 
as well as analogous cedents corresponding to the equality axioms
(the need to allow the arbitrary set of formulas $\Gamma$ to appear in axioms arises because there is no weakening rule). 
The most important rules from our perspective are the rules for introducing conjunctions and quantifiers: 
\begin{multicols}{3}
\begin{prooftree}
  \AxiomC{$\Gamma, \psi_1$}
  \AxiomC{$\Gamma, \psi_2$}
  \RightLabel{\scriptsize{$(\land)$},}
  \BinaryInfC{$\Gamma, \psi_1 \land \psi_2$}
 \end{prooftree}

\begin{prooftree}
  \AxiomC{$\Gamma, \psi(t)$}
  \RightLabel{\scriptsize{$(\exists)$},}
  \UnaryInfC{$\Gamma, \exists w \, \psi$}
 \end{prooftree}

 \begin{prooftree}
  \AxiomC{$\Gamma, \psi(a)$}
  \RightLabel{\scriptsize$(\forall)$,}
  \UnaryInfC{$\Gamma, \forall w\, \psi$}
 \end{prooftree}  
\end{multicols} 
\noindent where in the $(\exists)$ rule $t$ must be a term that is substitutable for $w$ in $\psi$,
and in the $(\forall)$ rule $a$ must be an eigenvariable, 
i.e.~a free variable that does not appear anywhere in the conclusion
of the rule. There are also natural disjunction introduction rules and the cut rule.

We may assume that $\ind\Sigma_1$ is axiomatized by finitely many sentences $\gamma_1,\ldots,\gamma_n$, where each $\gamma_i$ has the form 
\begin{multline*}
\forall \bar v\, \exists x\, \exists \bar y\, \forall \bar z\,\forall \bar z'\,  \\
[x<v_{1}\wedge [\neg \delta_i(0, \bar z, \bar v) \lor \delta_i(v_1, \bar y, \bar v) \lor (\delta_{i}(x,\bar y,\bar v) \land \neg \delta_i(x\!+\!1,\bar z', \bar v))]],
\end{multline*}
with $\delta_i$ bounded. (Using a different finite axiomatization would shorten proofs in $\ind\Sigma_1$ by at most a constant additive factor, and using the typical axiomatization of $\ind\Sigma_1$ as a scheme would shorten proofs at most polynomially.)

By the cut elimination theorem, which formalizes in (a fragment of) $\ind \Sigma_1$,
if there is an inconsistency proof from $\ind \Sigma_1$ of size at most $x$,
then for some fixed polynomial $p$ there is a cut-free proof of the cedent
\[\neg \gamma_1,\ldots,\neg \gamma_n,\]
of size at most ${\exp_{p(x)}(2)}$. Working in $\ind \Sigma_1$,
let $k$ be such that every infinite $\Delta_2$-set contains at least ${\exp_{p(k)}(2)}$
elements. Let $m$ stand for ${\exp_{p(k)}(2)}$.
We will prove that there is no cut-free proof
of $\neg \gamma_1,\ldots,\neg \gamma_n$ of size at most $m$,
which will imply $\Con_k(\ind\Sigma_1)$.

Assume to the contrary that there is such a cut-free proof,
and let the lines of the proof be $C_1,\ldots,C_\ell$;
note that $\ell \le m$. 
For each $j = 1, \ldots, \ell$, let the \emph{negations of} the formulas
in $C_j$ be $\xi_{j,1}, \ldots, \xi_{j,r_j}$. 
Note that $r_\ell = n$, each $\xi_{\ell,i}$ is $\gamma_i$,
and, by the subformula property of cut-free proofs, each $\xi_{j,r}$
is a subformula of one of the $\psi_i$'s.
As usual in such a context, we regard
$\varphi(t)$ as a subformula of $\mathrm{Q}x\,\varphi$
for $\mathrm{Q}$ a quantifier.

Define an infinite sequence of numbers by: 
\begin{align*}
d_0        & = 0, \\
d_{j+1} & = \textrm{least } d > d_j \textrm{ s.t., if } u \textrm{ is the smallest number s.t.~each term}  \\
             & ~~~~ \textrm{with }{\le}~m \textrm{ symbols evaluated on arguments }{\le}~d_j \textrm{ has value }{\le}~u,  \\
             & ~~~~ \textrm{then }d \ge u \textrm{ and, for each }i =1,\ldots,n, \textrm{ each } \bar v \textrm{ and } x: \\ 
             & ~~~~ \!\max(x,\max(\bar v)) \le {u} \land \exists \bar y\, \delta_i(x,\bar y, \bar v)
             \Rightarrow \exists \bar y\,(\max(\bar y) \le d \land \delta_i(x,\bar y, \bar v)).
\end{align*}
Let $D$ consist of all numbers that appear as some $d_j$. Note that provably in $\ind \Sigma_1$, both $D$ and the complement of $D$ are $\Sigma_2$-definable, so $D$ is a $\Delta_2$-set, and $D$ is infinite. 
By our assumption, there exists an $\ell$-element finite subset of $D$.
W.l.o.g., we may assume that the elements of this subset are $d_0, \ldots, d_{\ell-1}$.

We claim that the following statement $\eta(s)$ can be proved by $\Pi_1$ induction on $s = 0,\ldots,\ell\!-\!1$:
\begin{center}
\begin{minipage}{0.9\textwidth}
``there exist $j \le \ell - s$ and an assignment $\alpha$ of values $\le d_{s}$ to the free variables in $C_{j}$ such that, for every $\xi_{j,r}$ that is $\Sigma_2$, there is an assignment of values $\le d_{s}$ to the variables (if any) in the unbounded existential quantifier block of $\xi_{j,r}$ that together with $\alpha$ makes the $\Pi_1$ part of $\xi_{j,r}$ satisfied.''
\end{minipage}
\end{center}
Note that $\eta(s)$ is indeed a $\Pi_1$ statement (provably in $\bd\Sigma_1 + \exp$), because all the quantifiers preceding the definition of satisfaction for $\Pi_1$ formulas are bounded. Moreover, $\eta(0)$ is true, because it is witnessed by $j = \ell$ and the empty assignment, while $\eta( \ell-1)$ is false, because $C_1$ has to be a logical axiom, so an assignment witnessing the statement at $j = 1$ would have to satisfy two mutually contradictory quantifier-free formulas or falsify an equality axiom. Therefore, if the induction step goes through for $\eta(s)$, we obtain the required contradiction.

The induction step splits into cases depending on the rule used to derive $C_j$, where $j$ witnesses $\eta(s)$.
We consider the nontrivial cases, namely the ones corresponding to $(\land), (\exists)$, and $(\forall)$ inferences.

If $C_j$ was derived using the $(\land)$ rule, then $C_j$ is $\Gamma, \psi_1 \land \psi_2$, where $\psi_1 \land \psi_2$
is the (necessarily $\Delta_0$) principal formula of the inference used to derive $C_j$. Take the assignment $\alpha$ witnessing $\eta(s)$ at $j$, and let $j' < j$ be such that $C_j$ is $\Gamma, \psi_b$ for $b \in \{1,2\}$ such that $\alpha$ satisfies $\neg \psi_b$. This $j'$ and the unchanged assignment $\alpha$ witness $\eta(s+1)$.

If $C_j$ was derived by an $(\exists)$ inference, then $C_j$ is $\Gamma, \exists w\, \psi$ and $C_{j'}$ is 
$\Gamma, \psi(t)$ for some $j'<j$. 
In this case, we first extend a given assignment $\alpha$ witnessing $\eta(s)$ at $j$ 
to an assignment $\alpha'$ by letting all variables that are free in $C_{j'}$ but not in $C_j$ have value $0$. 
If $\psi(t)$ is not $\Pi_2$ (in which case $\neg \psi(t)$ is a $\Pi_3$ but not $\Sigma_2$ 
subformula of one of the induction axioms $\gamma_i$)
or $\exists w\,\psi$ is $\Sigma_1$ (in which case $\forall w\, \neg \psi$ is $\Pi_1$ and satisfied under $\alpha$,
so $\neg \psi(t)$ is satisfied under $\alpha'$),
this is all we need to do in order to ensure that $j', \alpha'$ witness $\eta(s+1)$.
The remaining case is when $\psi(t)$ is $\Pi_2$ but $\exists w\, \psi$ is not.
In that situation, $\neg\psi(t)$ arises from one of the $\gamma_i$'s by deleting the initial universal quantifier block 
and substituting some terms $\bar t$ for the variables $\bar v$ appearing in that block. 
We know that $\alpha'$ satisfies $\neg\psi(t)$ (because $\gamma_i$ is true),
but we also have to argue that we can witness the existential quantifiers $\es{x}{v_1}\exists \bar y$ 
in $\neg\psi(t)$ by numbers below $d_{s+1}$.
However, we know that we can find a value for $x$ below $\alpha'(t_1)$,
which is the value of a term with at most $m$ symbols on arguments below $d_s$. 
Thus, by the definition of $d_{s+1}$, we can also find values for $\bar y$ corresponding to $x$
in such a way that the maximum of these values is at most $d_{s+1}$. 

Finally, if $C_j$ was derived by a $(\forall)$ inference, then $C_j$ is $\Gamma, \forall w\, \psi$ and $C_{j'}$ is 
$\Gamma, \psi(a)$ for some $j'<j$ and some variable $a$ not appearing in $C_j$. 
Let $\alpha$ be an assignment witnessing $\eta(s)$ at $j$.
There are two subcases to consider, depending on whether $\exists w\, \neg \psi$
is an unbounded $\Sigma_2$ formula or a $\Delta_0$ formula. In the former case,
we know from the inductive assumption that $\alpha$ satisfies $\exists w\, \neg \psi$
and that there is a number $e \le  d_s$ witnessing the quantifier $\exists w$.
Then $j'$ and the assignment $\alpha \cup \{a:=e\}$ witness that $\eta(s+1)$ holds.
In the latter case, we know that $\alpha$ satisfies $\exists w\, \neg \psi$,
and we also know that any number $e$ witnessing the quantifier $\exists w$ must be bounded
by the value of a term appearing in $C_j$ (thus having at most $m$ symbols)
evaluated at elements of the range of $\alpha$, all of which are below $d_s$. 
By definition of $d_{s+1}$, this means that $e \le d_{s+1}$,
so again $j'$ and $\alpha \cup \{a:=e\}$ witness that $\eta(s+1)$ holds.
\end{proof}

\end{document}